\documentclass[11pt, oneside]{article}

\usepackage{amsmath, amssymb, amsthm}
\usepackage{graphicx}
\usepackage{xcolor}
\usepackage[font=small,labelfont=bf, width=.85\textwidth]{caption}
\usepackage{geometry}
\usepackage{comment}
\usepackage{enumerate}
\newtheorem{theorem}{Theorem}[section]

\newtheorem{conjecture}[theorem]{Conjecture}

\newtheorem{ques}[theorem]{Question}
\newtheorem{claim}{Claim}

\newenvironment{proofc}{\begin{proof}[Proof of Claim]}{\end{proof}}

\newcommand{\G}{\mathcal{G}}
\newcommand{\totdom}{\bar{\gamma}}
\newcommand{\xee}{\mathcal{X}}
\newcommand{\yee}{\mathcal{Y}}
\newcommand{\tee}{\mathcal{T}}
\newcommand{\see} {\mathcal{S}}
\newcommand{\wee} {\mathcal{W}}
\newcommand{\mee}{\mathcal{M}}
\newcommand{\comp}{\textnormal{comp}}

\begin{document}

\title{A reduction of the ``cycles plus $K_4$'s'' problem}
\author{
Aseem Dalal\footnote{ Indian Institute of Technology Delhi, Department of Mathematics, Delhi India.
  Email: {\tt aseem.dalal@gmail.com}.}
\qquad
Jessica McDonald\footnote{Auburn University, Department of Mathematics and Statistics, Auburn U.S.A.
  Email: {\tt mcdonald@auburn.edu}.
	Supported in part by Simons Foundation Grant \#845698  }
 \qquad
Songling Shan\footnote{Auburn University, Department of Mathematics and Statistics, Auburn U.S.A.
  Email: {\tt szs0398@auburn.edu}.
	Supported in part by NSF grant
 	DMS-2345869.}
}

\date{}

\maketitle

\begin{abstract}
Let $H$ be a 2-regular graph and let
$G$ be obtained from $H$ by gluing in vertex-disjoint copies of $K_4$.   The ``cycles plus $K_4$'s'' problem is to show that $G$  is 4-colourable; this is a  special case of the \emph{Strong Colouring Conjecture}. In this paper we reduce the ``cycles plus $K_4$'s'' problem to a  specific 3-colourability problem. In the 3-colourability problem, vertex-disjoint triangles are glued (in a limited way) onto a disjoint union of triangles and  paths of length at most 12, and we ask for 3-colourability of the resulting graph.
\end{abstract}

\vspace*{.3in}

\section{Introduction}

In this paper all graphs are assumed to be simple, unless explicitly
stated otherwise. The reader is referred to \cite{WestText} for
standard terminology.

Given vertex-disjoint graphs
$G_1, \ldots, G_q$, $H$ with
$|\bigcup _{1\leq i\leq q} V(G_i)|\leq |V(H)|$, we \emph{glue
  $G_1, \ldots, G_q$ onto $H$} by defining an injective function
$f:\cup_{1\leq i\leq q} V(G_i)\rightarrow V(H)$, and then forming a
new graph $G$ with $V(G) = V(H)$ and $E(G) = E(H) \cup \bigcup_{1\leq i \leq q} E_i,$
where
$E_i  =\{f(a)f(b): ab\in E(G_i), f(a)f(b)\not\in E(H)\}.$
The graph $G$ is said to have been obtained from $H$ by \emph{gluing
  in} $G_1, \ldots, G_q$.  Consider the following question:

\begin{ques}\label{qtri} Suppose that $H$ is a graph with $\Delta(H)\leq 2$, and suppose that $G$ is obtained from $H$ by gluing on vertex-disjoint triangles. When is $\chi(G)\leq 3$?
\end{ques}

For Question \ref{qtri}, if $H$ contains a $C_4$, then certainly a $K_4$ may be created in $G$ which would make $\chi(G)\not\leq 3$. Having $C_4$'s in $H$ is not the only thing that could go wrong however: Fleischner and Stiebitz \cite{fs-remarks} found an infinite family of examples where $H$ does not contain any $C_4$ components, but where $\chi(G)\not\leq 3$, answering a question of Erd\H{o}s ~\cite{erdos-fav}. The smallest of Fleishner and Stiebitz's examples has $H=C_5\cup C_{10}$; other known graphs $H$ with $\Delta(H)=2$ but which can yield negative answers to Question \ref{qtri} (i.e. there is a way to glue on vertex-disjoint triangles that gives a 4-chromatic graph) include $H=C_3\cup C_6$ (\"{O}hman \cite{Ohman}) and $H=C_7\cup K_1\cup K_1$ (Sachs, see \cite{fs-remarks}); for much more, see \"{O}hman. On the other hand, Fleishner and Stiebitz \cite{FSct} (and later Sachs \cite{Sa}) famously provided a positive answer to Question \ref{qtri}: they proved that if $H$ is a single cycle then the $G$ obtained is always 3-colourable. This ``cycle plus triangles'' theorem answered another question of Erd\H{o}s, which also had origins in the work of Du, Hsu, and Hwang \cite{DHH} (see \cite{FSct} for more history).

If we hope for a positive answer to Question \ref{qtri} for some $H$ which has multiple components, it seems wise to avoid cycles of length at least four in $H$, and instead consider $H$ to be a disjoint union of paths and triangles. Here there are two extremes, both of which always yield positive answers to Question \ref{qtri}. If $H$ is a disjoint union of paths, then by joining all these paths into a cycle, we form a ``cycle plus triangles'' graph $G$ which is 3-colourable (by Felishner and Stiebitz \cite{FSct}). On the other hand if $H$ is a disjoint union of triangles, then we can form a 3-regular auxiliary bipartite graph which describes the interaction between triangles of $H$ and added triangles, and the 3-edge-colourability of this graph implies the 3-colourability of $G$. Things become more difficult when glued-in triangles join both path and triangle components in $H$. For $G, H$ as in Question \ref{qtri}, we say that in $G$, two triangles $T, T'$ in $H$ are \emph{path-linked} if there exists at least one glued-in triangle $Y$ with $T\cap Y, T'\cap Y\neq \emptyset$, and the third vertex of $Y$ is on some path in $G$ of length at least two. If such gluings are strictly limited -- so that in $G$, every $H$-triangle is path-linked to at most one other $H$-triangle, then we conjecture that $G$ is 3-colourable. In fact, we also restrict our path components in $H$ to length at most 12 to eliminate other potential problems; note that if all $H$-paths have length at most 2 then our auxiliary bipartite argument above again works to show 3-colourability.

\begin{conjecture}\label{Conj3col} Let $H$ be a graph which is a disjoint union of triangles and paths of length at most 12, and let $G$ be obtained from $H$ by gluing on vertex-disjoint triangles. Moreover, suppose that in $G$, every $H$-triangle is path-linked to at most one other $H$-triangle. Then $G$ is 3-colourable.
\end{conjecture}

Given all its restrictions, we hope that Conjecture \ref{Conj3col} may be approachable. Our main result in this paper is that the truth of Conjecture \ref{Conj3col} would imply a seemingly very difficult conjecture, the ``cycles plus $K_4$'s problem'' that we state here now.

\begin{conjecture}\label{scc2} Let $H$ be a graph with $\Delta(H)\leq 2$ and let $G$ be obtained from $H$ by gluing on vertex-disjoint copies of $K_4$. Then $\chi(G)\leq 4$.
\end{conjecture}

Note that Conjecture \ref{scc2}, as stated above, is not just ``cycles plus $K_4$'s'', but rather it allows paths in the base graph as well.
However it is easy to see that it suffices to prove Conjecture \ref{scc2} for 2-regular $H$ (other $H$ being subgraphs of these); hence the nickname. Our main result of this paper is concretely the following.

\begin{theorem}\label{reduc} Conjecture \ref{Conj3col} implies Conjecture \ref{scc2}.
\end{theorem}

While we reduce Conjecture \ref{scc2} to Conjecture \ref{Conj3col} in this paper, we do not show that they are equivalent. However, an equivalence cannot be so far away from the truth. If Conjecture \ref{scc2} holds, then by deleting any one of the four colour classes we get a 3-colourable graph $G'$, and $G'$ is obtained by gluing vertex-disjoint triangles onto a graph $H'$ which is a disjoint union of paths and cycles.

To give some context for Conjecture \ref{scc2}, we must mention the \emph{strong chromatic number} of a graph $H$, denoted $s\chi(H)$, which was
introduced independently by Alon \cite{Al} and Fellows \cite{Fel}
in the late 1980's. Skipping this definition, let us just say that the  \emph{Strong Colouring Conjecture} posits that $s\chi(H)\leq 2\Delta(H)$; exact attribution of the conjecture is tricky, but the 1995 book of Jensen and Toft \cite{JT} (Section 4.14) has more on the early history of strong colouring.
Support for the Strong Colouring Conjecture has been given by numerous papers, eg. Haxell \cite{Hax04}\cite{Hax08}, by Aharoni, Berger and Ziv~\cite{ABZ}, Axenovich and Martin \cite{AM}, Johansson, Johansson and Markstr\"{o}m \cite{JJM}, and Lo and Sanhueza-Matamala~\cite{LS}. On the other hand, while the truth of the Strong Colouring Conjecture is trivial for $\Delta(H)=1$, (where it
asks essentially for the union of two matchings to be bipartite), the  $\Delta(H)=c$ case is open for all constants $c\geq 2$. In the most glaring open case of $\Delta(H)=2$, a result of Haxell \cite{Hax04} says that $s\chi(H)\leq 5$, but the conjectured upper bound is 4. In fact, the $\Delta(H)=2$ case of the Strong Colouring Conjecture is precisely Conjecture \ref{scc2} (see e.g. \cite{MP}). The only cases of Conjecture \ref{scc2} known to be true is when $H$ has at most one odd cycle of length exceeding $3$, or $H$ has at most $3$ triangles (McDonald and Puelo \cite{MP}). It seems that a new approach is needed to make a breakthrough on this problem, and our hope is that the reduction to Conjecture \ref{Conj3col} may help.

This paper proceeds as follows. The following section discusses so-called
\emph{independent sets of representatives} (ISRs), and states two results of   Haxell \cite{Hax95}\cite{Hax01} that will be needed for our main proof. We also prove a result about combining two ISRs into one, which generalizes a prior theorem of the second author and Puleo \cite{MP}, and may be of independent interest. The third section of the paper contains our proof of Theorem \ref{reduc}.

\section{Independent Sets of Representatives}

If $H$ is a graph and let $V_1, \ldots, V_n$ be
disjoint subsets of $V(H)$.  An \emph{independent set of representatives (ISR)} of $\{V_1, \ldots, V_n\}$ in $G$ is set $R\subseteq V(G)$ such that $R$ is independent in $G$ and $R$ contains exactly one vertex from each set $V_i$. If $R\subseteq V(G)$ is independent in $G$ and contains at most one vertex from each set $V_i$, then $R$ is said to be a \emph{partial ISR}; $R$ is said to \emph{hit} those $V_i$ for which it contains a representative. Note that if $R$ is a partial ISR of $\{V_1, \ldots, V_n\}$ in $G$ then it is an ISR of
$\{V_i: \textrm{$R$ hits $V_i$}\}$.

Haxell has proved the following.

\begin{theorem}[Haxell~\cite{Hax01}]\label{ISR4}
  If $H$ is a graph with $\Delta(H)\leq 2$ and $V_1, \ldots, V_n$
  are disjoint subsets of $V(H)$ with each $|V_i| \geq 4$,
  then $(V_1, \ldots, V_n)$ has an ISR.
\end{theorem}

To state a second result of Haxell, we need a few additional definitions. To this end, a \emph{total dominating set} in a graph $G$ is a set of vertices $X$
such that every vertex in $G$ is adjacent to a vertex in $X$. In
particular, every vertex of $X$ must also have a neighbor in $X$. The
\emph{total domination number} of $G$, written $\totdom(G)$, is the
size of a smallest total dominating set; if $G$ has isolated vertices,
then by convention we set $\totdom(G) = \infty$. Note that by this definition, every (nonempty) graph has total domination number at least two.  Given a graph $H$ and disjoint subsets $V_1, \ldots, V_n \subset V(H)$, for each $S \subset [n]$ we define a subgraph $H_S$ by taking the subgraph induced by the vertex set $\bigcup_{i \in S}V_i$ and deleting all edges in $H[V_i]$ for every $i\in S$.

\begin{theorem}[Haxell~\cite{Hax95}]\label{domISR}
  Let $H$ be a graph and let $V_1, \ldots, V_n$ be disjoint subsets of
  $V(H)$. If, for all $S \subset [n]$, we have $\totdom(H_S) \geq 2|S| - 1$, then $(V_1, \ldots, V_n)$ has an ISR.
\end{theorem}

It is worth noting we have stated Theorem \ref{domISR} as it appeared in \cite{MP} -- the original formulation was in terms of hypergraphs.

Haxell's Theorems \ref{ISR4} and \ref{domISR} are both about the existence of a single ISR. We now prove a theorem about combining two ISRs into one. To this end, let $\xee = \{X_1, \ldots, X_p\}$ and $\yee = \{Y_1, \ldots, Y_q\}$ be two collections of pairwise disjoint subsets of $V(G)$, and suppose that $R_{\xee}, R_{\yee}$ are ISRs of $\xee,\yee$, respectively, in $G$. If $G[R_{\xee}\cup R_{\yee}]$ is an independent set of vertices, then we view this union $R_{\xee}\cup R_{\yee}$ as the successful combination of two ISRs into one, since it is simultaneously an ISR of both $\xee$ and $\yee$. In fact, we would be just as happy if we could find $R\subseteq R_{\xee}\cup R_{\yee}$ that was simultaneously an ISR of both $\xee$ and $\yee$. This won't always be possible, but we will look to delete vertices from $R_{\xee}\cup R_{\yee}$ so that the resulting $R$ is independent and is somehow close to hitting all the sets in both $\xee$ and $\yee$. We say that an edge $e\in G[R_{\xee}\cup R_{\yee}]$ is a \emph{$\xee$-edge} (\emph{$\yee$-edge}) if both endpoints of $e$ are in $X_i$ ($Y_i$) for some $X_i\in \xee$ ($Y_i\in \yee$); note that it is possible for an edge to be both a $\xee$-edge and a $\yee$-edge. We define $E_{\xee\yee}$ to be all those edges of $G[R_{\xee}\cup R_{\yee}]$ that are neither $\xee$-edges nor $\yee$-edges.

\begin{theorem}\label{thm:combine-isr-gen}
Let $G$ be a graph. Suppose that $R_\xee$ is an ISR of $\xee$ in $G$ and $R_\yee$ is an ISR of $\yee$ in $G$. For all $X\in \xee$, denote by $v_X$ the representative of $X$ in $R_\xee$. Then $G$ has an independent set $R \subseteq R_\xee \cup R_\yee$ that is an ISR of $\yee$ and such that for every $X \in \xee$:
\begin{itemize}
    \item[(a)] if $v_X$ is not incident to any $E_{\xee\yee}$-edges, then $R$ hits $X$, and;
    \item[(b)] if $v_X$ is incident to at least one $E_{\xee\yee}$-edge, then $R$ hits $X \cup \{w\}$ for some $E_{\xee\yee}$-edge $wv_X$.
\end{itemize}
\end{theorem}

Before we proceed to the proof of Theorem \ref{thm:combine-isr-gen}, it may be remarked that if $E_{\xee\yee} = \emptyset$, then Theorem \ref{thm:combine-isr-gen} provides an independent set $R$ which is an ISR of both $\xee$ and $\yee$. In this way Theorem \ref{thm:combine-isr-gen} is a generalization of Lemma 3.3 of the second author and Puleo in \cite{MP} (their hypothesis that $(\xee, \yee)$ is an ``admissible-pair'' in $G$ implies that $E_{\xee\yee}=\emptyset$). Let us now prove Theorem \ref{thm:combine-isr-gen}.

\begin{proof}[Proof of Theorem \ref{thm:combine-isr-gen}]
Initially, let $R_0 = R_{\xee} \cup R_{\yee}$. The set $R_0$
clearly hits every $X_i \in \xee$ and $Y_j \in \yee$ in $G$, but as there may be edges between $R_{\xee}$ and $R_{\yee}$, the set $R_0$ may not be independent. We will describe an algorithm for iteratively deleting
vertices from $R_0$ in order to obtain an independent subset of $R_0$ which still hits every $Y\in \yee$ and also meets conditions (a) and (b) for each $X\in \xee$.  First however, let us prove the following claim.

\begin{claim}\label{claim:edgedeg}
Every vertex of $R_{\xee}$ is incident to at most one $\yee$-edge, and every vertex
of $R_{\yee}$ is incident to at most one $\xee$-edge.
\end{claim}
\begin{proofc}
  Suppose that $u \in R_{\xee}$ and that $uv_1$, $uv_2$ are two different
  $\yee$-edges incident to $u$. Then $v_1, v_2 \in R_{\yee}$ with
  $\{u, v_1\} \subseteq Y_1$ and $\{u, v_2\} \subseteq Y_2$ for
  some $Y_1, Y_2 \in \yee$. Since the sets in $\yee$ are
  pairwise vertex-disjoint, this implies that $Y_1 = Y_2$. But since $v_1, v_2\in R_{\yee}$, we must have $Y_1\neq Y_2$, contradiction. The same argument, interchanging the roles of $\xee$ and $\yee$, proves the claim about
$\xee$-edges.
\end{proofc}

Our algorithm defines a sequence of vertex sets $R_0, R_1, \ldots$ starting with
$R_0 = R_{\xee} \cup R_{\yee}$. Given some set $R_{i}$, we
either produce a new set $R_{i+1}$ and proceed to the next round, or
produce the final set $R$, via the following algorithm.\\

\noindent{\bf Step 1.} If there is a vertex $v\in R_{\yee}$ which is isolated in $G[R_i]\setminus E_{\xee\yee}$ but is incident to least one $E_{\xee\yee}$-edge in $G[R_i]$, then: form $R_{i+1}$ from $R_i$ by deleting all vertices that are adjacent to $v$ in $G[R_i]$, and then go back to the start of {\bf Step 1}.

\noindent{\bf Step 2.} If there is a vertex $v\in R_{\yee}$ which
has degree $1$ in $G[R_i]\setminus E_{\xee\yee}$, and this one edge is a $\xee$-edge, then: form $R_{i+1}$ from $R_i$ by deleting all vertices that are adjacent to $v$ in $G[R_i]$, and then go back to the start of {\bf Step 1}.

\noindent{\bf Step 3.} If there is a vertex $v\in R_{\xee}$ which has degree $1$ in $G[R_i]$, and this one edge is a $\yee$-edge, then: form $R_{i+1}$ from $R_i$ by deleting the one vertex in $R_i$ that is adjacent to $v$ in $G[R_i]$, and then go back to the start of {\bf Step 1}.

\noindent{\bf Step 4.} Otherwise, obtain $R$ from $R_i$ by deleting every vertex of $R_{\yee} \cap R_{i}$ that has positive degree in $G[R_{i}]$, and then terminate. \\

We call any vertex $v$ found in {\bf Steps $1$--$3$} above a \emph{dangerous} vertex (for $R_i$), and only reach Step 4 when $R_i$ has no dangerous vertices. Note that vertices which were not initially dangerous for $R_0$ may become dangerous for some later $R_i$ as their neighbors are deleted. However, once a dangerous vertex $v$ is found in one of {\bf Steps 1--3}, all of $v$'s neighbours in $G[R_i]$ are deleted, and hence $v$ will be isolated in $G[R_j]$ for all $j\geq i$, and will remain until the end of the algorithm and be a member of our terminal $R$.

The algorithm always terminates, since
$|R_{i+1}| < |R_{i}|$ whenever $R_{i}$ has a dangerous
vertex. Moreover, the terminal set $R$ is always independent due to {\bf Step 4}.  It remains to show that $R$ hits every set $Y \in \yee$ and also meets conditions (a) and (b) for each $X\in \xee$.

Consider any set $Y \in \yee$. Let $w$ be the representative of $Y$ in $R_\yee$. If $w \in R$, then $R$ hits $Y$ and we are done. Therefore, we suppose that $w$ was deleted by the algorithm. If $w\in R_{\yee}\cap R_{\xee}$ then $w$ is isolated in $G[R_0]$ and never gets deleted. So $w\in R_{\yee}\setminus R_{\xee}$ and hence must have been deleted in {\bf Step 3} or {\bf Step 4}.

First suppose that $w$ was deleted in {\bf Step 3} due to $v$ being a dangerous vertex for $R_i$. Then, based on our earlier comments,  $v \in R$. Since the deletion of $w$ happened in {\bf Step 3}, we know that $vw$ is a $\yee$-edge, so $v,w\in Y'$ for some $Y'\in \yee$. But the sets in $\yee$ are disjoint and $w\in Y$, so in fact $Y'=Y$, and $v$ ensures that $R$ hits $Y$.

Now assume that $w$ was deleted in {\bf Step 4} after determining that $R_i$ has no dangerous vertices. We know that $w$ is not isolated in $G[R_i]$, since otherwise it would not have been deleted in {\bf Step 4}. So since $w$ is not dangerous for $R_i$ according to {\bf Step 1}, it must be incident to at least one $\xee$-edge or $\yee$-edge in $G[R_i]$. In fact, since $w$ is not dangerous for $R_i$ according to {\bf Step 2}, and given Claim \ref{claim:edgedeg}, $w$ must have either degree one or two in $G[R_i]$; in the former case its incident edge is a $\yee$-edge and in the latter case it is incident to  both a $\xee$-edge and a $\yee$-edge (which are distinct). In either case we know that there is a $\yee$-edge incident to $w$ in $G[R_i]$, say $wu$. Since $w\in Y$ we get that $u \in Y$ as well. Since $w$ is the lone representative for $Y$ in $R_{\yee}$ (or since $u\sim w$ and $R_{\yee}$ is independent) we know that $u\not\in R_{\yee}$, so $u$ is not deleted in this {\bf Step 4}. But since this is the very last step of our algorithm, no subsequent step could have deleted $u$ either, so $u \in R$ at the
end. Hence $R$ hits $Y$.

Now consider any $X \in \xee$. If $v_X \in R$, then $R$ hits $X$ (and thus hits $X \cup \{u\}$ for every $E_{\xee\yee}$-edge $uv_X$ in $G$) and we are done. If $v_X \in R_{\xee}\cap R_{\yee}$ then $v_X$ is isolated in $G[R_0]$ and never gets deleted. Therefore, we suppose $v_X$ was deleted by the algorithm and $v_X \in R_{\xee}\setminus R_{\yee}$. Hence, $v_X$ must have been deleted in {\bf Step 1} or {\bf Step 2}.

First assume that $v_X$ was deleted in {\bf Step 1} due to $w$ being a dangerous vertex for $R_i$. Then, based on our earlier comments, $w \in R$. But then $v_X,w$ are joined by an \emph{$E_{\xee\yee}$-edge} and  $R$ hits $X \cup \{w\}$.

Now we may assume that $v_X$ was deleted in {\bf Step 2} due to $w$ being a dangerous vertex for $R_i$. Then, again based on our earlier comments, $w \in R$. We know that either $v_Xw$ is an \emph{$\xee$-edge} or $v_Xw$ is an $E_{\xee\yee}$-edge. In the latter case, $R$ hits $X \cup \{w\}$. In the former case, $w\in X$ and so $R$ hits $X$. Either way, conditions (a) and (b) are satisfied.
\end{proof}

\section{Proof of Theorem \ref{reduc}}

Within our proof of Theorem \ref{reduc}, we will find occasion to use the following classic result of Lov\'{a}sz \cite{Lov}  (see also Theorem 20 in \cite{Wood}).

\begin{theorem}[Lov\'{a}sz \cite{Lov}]\label{lov} Let $d, D$ be a non-negative integers and let $k=\lfloor\tfrac{D}{d+1}\rfloor+1$. If $G$ is a graph with maximum degree $D$, then $V(G)$ can be partitioned onto $k$ sets $V_1, \ldots V_k$ such that $G[V_i]$ has maximum degree at most $d$ for all $i\in\{1, 2, \ldots, k\}$.
\end{theorem}

Let us now proceed with our main proof.

\begin{proof}\emph{(Theorem \ref{reduc})}
Let $H$ be a graph with $\Delta(H)\leq 2$, and let $G$ be obtained from $H$ by gluing on vertex-disjoint copies of $K_4$. We show that, under the assumption that Conjecture \ref{Conj3col} is true, $\chi(G)\leq 4$. As previosuly discussed, we may assume that $H$ is 2-regular.

Let $X_1, \ldots, X_p$ be the vertex sets of the components of $H$ and
let $Y_1, \ldots, Y_q$ be the vertex sets of the added copies of $K_4$. Let $\xee = \{X_1, \ldots, X_p\}$ and let
  $\yee = \{Y_1, \ldots, Y_q\}$.
  Let $\tee\subseteq\xee$ correspond to those components of $H$ which are triangles, with $t=|\tee|$.
  For any $\wee\subseteq \xee$ define $G^{\wee}$ to be the the graph with vertex set $\tee$ and where two vertices are joined by an edge if the corresponding $X_i, X_j$ have the property:
  there exists (at least one) $Y\in \yee$ with $Y\cap X_i, Y\cap X_j , Y\cap X_k\neq \emptyset$ for some $X_k\in\xee$ corresponding to a cycle of length at least four.

Choose $\mee$ such that:
\begin{enumerate}
\item[(M1)] $\mee\subseteq \tee$;
\item[(M2)] $\Delta(G^{\mee})\leq 1$;
\item[(M3)] $|\mee|$ is maximum, subject to (M1) and (M2), and;
\item[(M4)] $|E(G^{\mee})|$ is minimum, subject to (M1), (M2), (M3).
\end{enumerate}
We can prove the following about the size of $\mee$.

\begin{claim}\label{Msize}
$|\mee| \ge \tfrac{t}{4}.$
\end{claim}
\begin{proofc}
Let $T\in \tee$ and consider its degree in $G^{\xee}$. Suppose $v\in T$ and $v\in Y\in \yee$. If $Y$ contributes to the degree of $T$ in $G^{\xee}$, then at least one of the four vertices in $Y$ must be a member of a non-triangle cycle in $H$. But in this case, $Y$ contains vertices from at most two different members of $\tee$, aside from $T$, and hence $Y$ contributes at most two to the degree of $T$ in $G^{\xee}$. Since $T$ has size three, this means that $T$ has degree at most $6$ in $G^{\xee}$.
Applying Theorem \ref{lov} to the graph $G^{\xee}$ with $D=6$ and $d=1$ gives a partition of $V(G^{\xee})$ into four parts, each of which induces a subgraph of maximum degree at most one. If we take the subset of $\tee$ corresponding to the largest of these parts, then it has at least $t/4$
elements.
\end{proofc}

We consider now the long cycles in $H$, which we intend to break into short pieces by deleting some vertices. In particular, we let $L=\{i: X_i\in \xee, |X_i|\geq 15\}$. For each $i\in L$, if $X_i$ consists of the vertex set of the cycle $(x_1, x_2, \ldots, x_{l_i})$, we let $p_i = \lfloor \tfrac{l_i}{5} \rfloor$ and define $X_i^* = \{x_{5(j)-4}: 1\leq j\leq p_i\}$. We then let $X_i^1, \ldots, X_i^{p_i}$ be the vertex sets of the $p_i$ cycle-segments created by deleting $X_i^*$, namely $X_i^j = \{x_{5j-3}, \ldots, x_{5j}\}$ for $1\le j \le p_i-1$ and $X_i^{p_i} = \{x_{5p_i-3}, \ldots, x_{l_i}\}$. Note that $|X_i^j| = 4$ for $1 \le j \le p_i-1$ and $4 \le |X_i^{p_i}|\le 8$.

    Let $G'$ be the graph formed from $G$ by deleting all the vertices in all the sets $X_i^*$, that is,  $G'= G \setminus \left(\bigcup_{i\in L} X_i^{*}\right)$. We obtain $\xee'$ from $\xee$ by replacing each $X_i$, $i\in L$, with the $p_i$ sets $X_i^1, \ldots X_i^{p_i}$. Let $\xee' = \{X'_1, X'_2, \ldots, X'_{p'}\}$, noting that $p'\geq p=|\xee|$. Note also that we still have $\mee\subseteq \xee'$, since triangles are not affected by the deletion process.

The goal in this next section of our proof will be to find an ISR of $\xee'\setminus \mee$ in the graph $G'$.
Since we are unconcerned with hitting the triangle parts, it will serve us to form an auxiliary graph $G''$ by taking the disjoint union of $G'$ along with $m= |\mee|$ copies of $K_m$.
For all $X'_i \not \in \mee$, define $X''_i=X_i'$
and otherwise define $X''_i$ to be $X_i'$ together with one vertex from each copy of $K_m$, chosen so that $X_1'', \ldots, X_{p'}''$ are disjoint, and let $\xee''$ be this last collection of sets. We will aim, through the next two claims, to get an ISR of $\xee''$ in $G''$ via Theorem \ref{domISR}. When restricted to $G'$, such an ISR would contain a representative from each set in $\xee'$ except possibly those in $\mee$. So we would indeed be able to get an ISR  of $\xee'\setminus \mee$ in the graph $G'$, as we have said we want.

In order to apply Theorem \ref{domISR} (to get an ISR of $\xee''$ in $G''$), we start by letting $S \subset [p']$, with $\see$ the set of $X''_i \in \xee''$ corresponding to $S$. Let $t_s=|\tee\cap \see|$.
We consider the graph $G''_S$ defined with respect to $X''_1, \ldots, X''_{p'}$ (as discussed prior to the statement of Theorem \ref{domISR}).
Note that the set of edges removed from $G''$ to make $G''_S$ are exactly the same as the set of edges removed from $G'$ to make $G'_S$, since each $X''_i$ is obtained from $X'_i$ by adding either an independent set or nothing.
So $G''_S$ is just the disjoint union of $G'_S$ and $m$ copies of $K_{|\see\cap \mee|}$;
in particular note that $G''_S=G'_S$ when $\see\cap \mee=\emptyset$.

\begin{claim}\label{compBound} Let $\comp(G'_S)$ be the number of components of $G'_S$. Then in general,
$\comp(G'_S)\geq \tfrac{1}{4}|V(G'_S)|$. Moreover,  if $\see\cap \mee =\emptyset$, then
$\comp(G'_S)\geq \tfrac{1}{4}\left(|V(G'_S)| + t_s\right).$
\end{claim}

\begin{proofc} Since all edges of $H$ are removed when forming $G'_S$, each component of $G'_S$ has size at most four, and we immediately get the $\tfrac{1}{4}|V(G'_S)|$ bound. Now suppose that $\see\cap \mee =\emptyset$.

Let $T\in \tee\cap\see$. Then $T\not\in \mee$, by assumption. We claim that in $G^{\xee}$ (note that this graph is formed prior to any deletions or breaking of long cycles), there are at least two triangles in $\mee$ that are each in some common glued $K_4$ with $T$.
Certainly we cannot have no edges, since then  $\tilde{\mee}= \mee\cup \{T\}$ has $|\tilde{\mee}| > |\mee|$, violating (M3).
So suppose for a contradiction that there is exactly one such edge in $G^{\xee}$, say from $T$ to $\tilde{T}\in \mee$. We know that $\tilde{T}$ has at most one edge in $G^{\xee}$ joining it to other triangles in $\mee$. If it actually has no such edges, then again $\tilde{\mee}= \mee\cup \{T\}$ satisfies (M1) and (M2) but $|\tilde{\mee}| > |\mee|$ violating (M3).
So $\tilde{T}$ must have exactly one edge in $G^{\xee}$ joining it to other triangles in $\mee$; now  $\tilde{\mee}= (\mee \setminus
\{\tilde{T}\}) \cup \{T\}$ satisfies (M1), (M2), (M3), but  $|E(G^{\tilde{\mee}})|<|E(G^{\mee})|$, violating (M4). So indeed, there are at least two edges from $T$ to $\mee$ in $G^{\xee}$. Let $E_T$ be such a pair of edges.

Either $E_T$ comes from one $Y\in \yee$ containing vertices from both $T$ and two different triangles in $\mee$, or it comes from two different $Y^1, Y^2\in \yee$, both of which contain a vertex from $T$ and a triangle in $\mee$. In all cases, by the definition of $G^{\xee}$, each of $Y, Y^1, Y^2$ must contain at least one vertex from a cycle of $H$ of size at least four; this cycle may or may not be in $\see$. Partition $\see\cap \tee$ into $\see_1, \see_2$ so that for a given $T\in \see\cap \tee$, $T\in \see_1$ if $E_T$ comes from one $Y\in \yee$, and $T\in \see_2$ if $E_T$ comes from two $Y^1, Y^2\in \yee$.

If $T\in \see_1$, the vertex in $T\cap Y$ is in a component of size at most 2 in $G_S$ (since $\mee\cap \see =\emptyset$).  If $T\in \see_2$,
the two vertices in $T\cap Y', T\cap Y''$  are both in components of size at most 3 in $G_S$ (since $\mee\cap \see =\emptyset$). However it could be that the third vertex in such a component (other than the vertex in $T$ and the vertex in the longer cycle) is another triangle in $\see_2$. Overall, for every triangle in $\see_1$ we may count one component of $G_S$ of size at most three (in fact size at most two), and for every triangle in $\see_2$ we may count one component of size at most three. In fact, all these vertices in triangles still exist in $G'_S$ (since we only delete vertices from long cycles), so the previous sentence remains true if we replace $G_S$ with $G'_S$. Since we already know that all components in $G'_S$ have size at most $4$, we get:
\begin{eqnarray*}
\comp(G'_S) &\geq& \tfrac{1}{4}(|V(G'_S)|-2|\see_1|-3|S_2|)+|S_1|+|S_2|\\
 &=& \tfrac{1}{4}|V(G'_S)|+\tfrac{1}{2}|S_1|+\tfrac{1}{4}|S_2|\\
 &=& \tfrac{1}{4}|V(G'_S)|+\tfrac{1}{2}|S_1|+\tfrac{1}{4}|S_2|\\
 &\geq& \tfrac{1}{4}|V(G'_S)|+\tfrac{1}{4}(|S_1|+|S_2|)=\tfrac{1}{4}(|V(G'_S)|+t_s).
\end{eqnarray*}
\end{proofc}

We will now use Claim \ref{compBound} as part of our proof of the following.

\begin{claim}\label{totdomGood} $\totdom(G''_S)\geq 2|S|$.
\end{claim}

\begin{proofc} We proceed by cases according to $|\see\cap \mee|$. If this value is one, then $G''_S$ contains isolates, and the result follows trivially.

Suppose first that $|\see\cap \mee| \geq 2.$ Then
$\totdom(K_{|\see\cap \mee|})=2$, and using the result of Claim \ref{Msize}, we get
\begin{equation}\label{totdom2}
\totdom(G''_S) = \totdom(G'_S) + m\cdot\totdom(K_{|\see\cap \mee|})=\totdom(G'_S)+ 2m \geq \totdom(G'_S)+ \tfrac{t}{2}
\end{equation}
Since every component of $G'_S$ has total domination number at least 2, by the first bound in Claim \ref{compBound}, we get
\begin{equation}\label{V2}\totdom(G'_S) \geq 2(|V(G'_S)|/4)=\tfrac{1}{2}|V(G'_S)|.
\end{equation}
On the other hand, all $X'_i \in \xee'$ have at least three vertices, with the only parts of size three corresponding to triangles in $H$ (recall that our broken cycle pieces have at least four vertices). It is possible that all $t$ of the triangle parts are included in $\see$, but even still,
\begin{equation}\label{V4}
|V(G'_S)| \geq 4|S| - t.
\end{equation}
Combining (\ref{totdom2}), (\ref{V2}), and (\ref{V4}), we get
$$\totdom(G''_S) \geq \tfrac{1}{2}(4|S|-t)+ \tfrac{t}{2}=2|S|.$$

We may now assume that $\see\cap \mee =\emptyset.$ Then
$\totdom(G''_S)=\totdom(G'_S)$. Since every component of $G'_S$ has total domination number at least 2, and by the second bound in Claim \ref{compBound}, we get
\begin{equation}\label{V0}\totdom(G'_S) \geq 2(\tfrac{1}{4}(|V(G'_S)|+t_s))=\tfrac{1}{2}|V(G'_S)|+ \tfrac{t_s}{2}.
\end{equation}
We know that exactly $t_s$  triangle parts are included in $\see$, so we get
\begin{equation}\label{V40}
|V(G'_S)| \geq 4|S| - t_s.
\end{equation}
Combining (\ref{V0}) and (\ref{V40}), we get
$$\totdom(G''_S)=\totdom(G'_S)\geq \tfrac{1}{2}(4|S|-t_s)+ \tfrac{t_s}{2}=2|S|.$$
\end{proofc}

By Claim \ref{totdomGood}, we know that $\totdom(G''_S) \geq 2|S|-1$, so we may finally apply Theorem \ref{domISR} to get an ISR of $\xee''$ in $G''$. As previously discussed, this in turn allows us to get an ISR of $\xee'\setminus \mee$ in $G'$, which we have said was our goal for this section of the proof. Let $\tilde{\xee}=\xee'\setminus \mee$, and call this last ISR $R_{\tilde{\xee}}$. In fact, $R_{\tilde{\xee}}$ is also an ISR of $\tilde{\xee}$ in $G$, since no edges are added between vertices of $R_{\tilde{\xee}}$ when moving back to $G$ from $G'$.

Since $\Delta(H)\leq 2$ and $|Y_i|\geq 4$ for all $i$, Theorem \ref{ISR4} guarantees that $\yee$ has an ISR in $H$, say $R_{\yee}$.

Since $R_{\yee}$ is an ISR of $\yee$ in $G$, and the $R_{\tilde{\xee}}$ is an ISR of $\tilde{\xee}$ in $G$, we can now apply Theorem \ref{thm:combine-isr-gen}. The result is a set $R$ that is independent in $G$, that is an ISR of $\yee$, and that hits  $\bigcup_{v \in V(C)} X_v$ for each component $C$ of the graph $\G_{\tilde{\xee}\yee}$.
Since $R$ hits every set in $\yee$, $G-R$ is obtained from some graph $H'\subseteq H$ by gluing on triangles.

\begin{claim}\label{Hprime}
$H'$ is a disjoint union of triangles and paths of length at most 12.
\end{claim}
\begin{proofc}
Consider a cycle in $H$ of length at most 14 not belonging to $\mee$, say represented by $X\in \xee$. Such a cycle is unaltered when moving from $G, \xee$ to $G', \xee'$ so $X\in \xee'$ and since $X \not \in \mee$, $X\in \tilde{\xee}$. Suppose $v$ is the representative of $X$ in $R_{\tilde{\xee}}$. Consider $\G_{\tilde{\xee}\yee}$. Since $v$ is the representative of $X$ in $\tilde{\xee}$, $X_v = X$. Now any vertex $w \in R_\yee$ adjacent to $v$ in $G$ either belongs to $X$ (implying that $vw$ is an $\tilde{\xee}$-edge) or both $w, v \in Y$ for some $Y \in \yee$ (implying that $vw$ is an $\yee$-edge). Therefore, by definition, there is no $E_{\tilde{\xee}\yee}$ edge incident to $v$. So by Theorem \ref{thm:combine-isr-gen}, $R$ hits $X$. This means that in $G-R$, all that remains of the cycle represented by $X$ in $H$ (i.e. its contribution to $H'$) is a subgraph of a path with at most 13  vertices (i.e. a path of length at most 12).

Consider now a cycle $(x_1, x_2, \ldots, x_{\ell_i})$ in $H$ where $\ell_i \ge  15$, say represented by  $X_i\in \xee$.   Suppose, for a contradiction, that that there exists a path segment $P$ of $\ell\geq 14$ vertices (so with length at least 13) which is a subgraph of this cycle and such that $V(P)\cap R = \emptyset$; without loss of generality suppose that $P=(x_1, \ldots, x_{\ell})$

 We know that $|X_i^j| = 4$ for $1 \le j \le p_i-1$ and $4 \le |X_i^{p_i}| \le 8$. Moreover, any two $X_i^j$s are separated (along the path on cycle $X_i$) by exactly one vertex and that belongs to $X_i^*$. Now since $\ell > 8+1+4=13$, $P$ contains a path segment $a X_i^{j}b$ such that $a, b \in X_i^*$ and  $X_i^j \in \xee'$. We know that $R_{\tilde{\xee}}$ uniquely hits $X_i^j$. Let $w_{j}$ be this representative for $X_i^{j}$, that is, $w_{j} = X_i^{j} \cap R_{\tilde{\xee}}$.

If there is no $E_{\tilde{\xee}\yee}$-edge incident to $w_{j}$, then by Theorem \ref{thm:combine-isr-gen} $R$ hits $X_i^j$, contradicting our assumption that $V(P) \cap R = \emptyset$. Therefore, there must be at least one $E_{\tilde{\xee}\yee}$-edge incident to $w_{j}$; let $v$ be a second endpoint of such an edge. Note that $v \in \{a,b\}$ because if not then $vw_{j}$ is either an $\tilde{\xee}$-edge or a $\yee$-edge. So, either $v = a$ and $w_{j}$ is the lowest-indexed member of $X_i^{j}$, or $v = b$ and $w_{j}$ is the highest-indexed member of $X_i^{j}$. By symmetry, we can assume that $v = a$. There cannot be any other $E_{\tilde{\xee}\yee}$-edges incident to $w_{j}$, since $|X_i^{j}|\geq 4$ means that $w_{j}$ is certainly not followed by $b$ on $P$. Then by Theorem \ref{thm:combine-isr-gen}, $R$ hits $X_i^{j} \cup \{a\} \subseteq V(P)$, contradiction.
\end{proofc}

Consider now the graph $H'$ in the context of Claim \ref{Hprime}. Note that every triangle in $H'$ is from $\mee$. By condition (M2), we know that $\Delta(G^{\mee})\leq 1$. In $H'$ this means that if $T_1$ is a triangle in $H'$, then there is at most one other triangle $T_2$ in $H'$ such that they are both joined in $G-R$ by some added $K_3$, say $Y'$, whose third vertex is from some cycle of length at least four in $H$.

In $G-R$, a pair of triangles in $H'$ are path-linked iff there is some added $K_3$ containing vertices from both triangles as well as from a path in $H'$ of length at least two.  Since paths in $H'$ of length at least two were all part of a cycle of length at least four in $H$ (paths in $H'$ coming from triangles in $H$ have length one), the above paragraph tells us that in $G-R$ every $H'$-triangle is path-linked to at most one other $H'$-triangle. So the truth of Conjecture \ref{Conj3col} would give a 3-colouring of $G-R$, and hence a 4-colouring of $G$.
\end{proof}

\bibliographystyle{amsplain}
\bibliography{bib}

\providecommand{\bysame}{\leavevmode\hbox to3em{\hrulefill}\thinspace}
\providecommand{\MR}{\relax\ifhmode\unskip\space\fi MR }
\providecommand{\MRhref}[2]{%
  \href{http://www.ams.org/mathscinet-getitem?mr=#1}{#2}
}
\providecommand{\href}[2]{#2}
\begin{thebibliography}{10}

\bibitem{ABZ}
Ron Aharoni, Eli Berger, and Ran Ziv, \emph{Independent systems of
  representatives in weighted graphs}, Combinatorica \textbf{27} (2007), no.~3,
  253--267.

\bibitem{Al}
N.~Alon, \emph{The linear arboricity of graphs}, Israel J. Math. \textbf{62}
  (1988), no.~3, 311--325.

\bibitem{AM}
Maria Axenovich and Ryan Martin, \emph{On the strong chromatic number of
  graphs}, SIAM J. Discrete Math. \textbf{20} (2006), no.~3, 741--747.

\bibitem{DHH}
D.~Z. Du, D.~F. Hsu, and F.~K. Hwang, \emph{The {H}amiltonian property of
  consecutive-{$d$} digraphs}, vol.~17, 1993, Graph-theoretic models in
  computer science, II (Las Cruces, NM, 1988--1990), pp.~61--63.

\bibitem{erdos-fav}
Paul Erd\H{o}s, \emph{On some of my favourite problems in graph theory and
  block designs}, vol.~45, 1990, Graphs, designs and combinatorial geometries
  (Catania, 1989), pp.~61--73 (1991).

\bibitem{Fel}
Michael~R. Fellows, \emph{Transversals of vertex partitions in graphs}, SIAM J.
  Discrete Math. \textbf{3} (1990), no.~2, 206--215.

\bibitem{fs-remarks}
H.~Fleischner and M.~Stiebitz, \emph{Some remarks on the cycle plus triangles
  problem}, The mathematics of {P}aul {E}rd\H os, {II}, Algorithms Combin.,
  vol.~14, Springer, Berlin, 1997, pp.~136--142.

\bibitem{FSct}
Herbert Fleischner and Michael Stiebitz, \emph{A solution to a colouring
  problem of {P}. {E}rd{\H{o}}s}, Discrete Math. \textbf{101} (1992), no.~1-3,
  39--48, Special volume to mark the centennial of Julius Petersen's ``Die
  Theorie der regul\"aren Graphs'', Part II.

\bibitem{Hax95}
P.~E. Haxell, \emph{A condition for matchability in hypergraphs}, Graphs
  Combin. \textbf{11} (1995), no.~3, 245--248.

\bibitem{Hax01}
\bysame, \emph{A note on vertex list colouring}, Combinatorics, Probability \&
  Computing \textbf{10} (2001), no.~4, 345.

\bibitem{Hax04}
\bysame, \emph{On the strong chromatic number}, Combin. Probab. Comput.
  \textbf{13} (2004), no.~6, 857--865.

\bibitem{Hax08}
\bysame, \emph{An improved bound for the strong chromatic number}, J. Graph
  Theory \textbf{58} (2008), no.~2, 148--158.

\bibitem{JT}
Tommy~R. Jensen and Bjarne Toft, \emph{Graph coloring problems},
  Wiley-Interscience Series in Discrete Mathematics and Optimization, John
  Wiley \& Sons, Inc., New York, 1995, A Wiley-Interscience Publication.

\bibitem{JJM}
Anders Johansson, Robert Johansson, and Klas Markstr\"{o}m, \emph{Factors of
  {$r$}-partite graphs and bounds for the strong chromatic number}, Ars Combin.
  \textbf{95} (2010), 277--287.

\bibitem{LS}
Allan Lo and Nicolás Sanhueza-Matamala, \emph{An asymptotic bound for the
  strong chromatic number}, Combinatorics, Probability and Computing (2017).

\bibitem{Lov}
L.~Lov\'{a}sz, \emph{On decomposition of graphs}, Studia Sci. Math. Hungar.
  \textbf{1} (1966), 237--238.

\bibitem{MP}
Jessica McDonald and Gregory~J. Puleo, \emph{Strong coloring 2-regular graphs:
  cycle restrictions and partial colorings}, J. Graph Theory \textbf{100}
  (2022), no.~4, 653--670.

\bibitem{Ohman}
Lars-Daniel \"{O}hman, \emph{Strong chromatic numbers of graphs with maximum
  degree 2}, unpublished manuscript, available at
  https://lars-daniel.se/papers/index.htm.

\bibitem{Sa}
H.~Sachs, \emph{Elementary proof of the cycle-plus-triangles theorem},
  Combinatorics, {P}aul {E}rd\H{o}s is eighty, {V}ol. 1, Bolyai Soc. Math.
  Stud., J\'{a}nos Bolyai Math. Soc., Budapest, 1993, pp.~347--359.

\bibitem{WestText}
Douglas~B. West, \emph{Introduction to graph theory}, 2 ed., Prentice Hall,
  September 2000.

\bibitem{Wood}
David~R. Wood, \emph{Defective and clustered graph colouring}, Electron. J.
  Combin. \textbf{DS23} (2018), 71.

\end{thebibliography}

\end{document}